\documentclass[paper=A4]{amsproc}

\usepackage{amsmath}
\usepackage{bbm}
\usepackage{amssymb}
\usepackage{amscd}
\usepackage{amsxtra}
\usepackage{amsfonts}
\usepackage{amsthm}
\usepackage{amsbsy}
\usepackage{algorithm}
\usepackage{algorithmic}
\usepackage{graphicx}
\usepackage{epstopdf}
\usepackage{graphicx}
\usepackage{epstopdf}
\usepackage{float}

\numberwithin{equation}{section} 
\numberwithin{figure}{section} 
\numberwithin{table}{section} 

\theoremstyle{definition}
\newtheorem{theorem}{Theorem}[section]
\newtheorem{proposition}[theorem]{Proposition}

\newtheorem{definition}[theorem]{Definition}

\newcommand{\R}{\mathbb R}
\newcommand{\1}{\mathbbm 1}

\newcommand{\Lg}{\mathcal L}

\begin{document}
\title{On Optimal Frame Conditioners}

\author{Chae Clark}
\address{Department of Mathematics\\University of Maryland, College Park}
\email{cclark18@math.umd}

\author{Kasso A. Okoudjou}
\address{Department of Mathematics\\University of Maryland, College Park}
\email{kasso@math.umd.edu}

\maketitle

\begin{abstract}
A (unit norm) frame is scalable if its vectors can be rescaled so as to result into a tight frame. Tight frames  can be considered optimally conditioned because the condition number of their  frame operators is unity.
In this paper we reformulate the scalability problem  as a convex optimization question. In particular, we present examples of various formulations of the problem along with numerical results obtained by using our methods on randomly generated frames. 
\end{abstract}

\section{Frames and scalable frames}
\subsection{Introduction} A finite frame for $\R^{N}$ is a set $\Phi=\{\varphi_{k}\}_{k=1}^{M}\subset\R^{N}$  such that there exist positive constants $0<A\leq B < \infty$ (referred to as the frame bounds) for which
\begin{equation*}\label{eq1}
A\|x\|_{2}^{2} \leq \sum_{k=1}^{M}|\langle x , \varphi_{k} \rangle|^2 \leq B\|x\|_{2}^{2}
\end{equation*}
for all $x\in\R^N$.
Given a frame $\Phi=\{\varphi_{k}\}_{k=1}^{M} \subset \R^N$,  we denote again by $\Phi$ the 
$N\times M$ matrix whose $k^{th}$ column is the vector $\varphi_k$.
The matrix $\Phi$ is the synthesis operator associated to the frame $\Phi$, and its transpose  $\Phi^{T}$ is the analysis operator of $\Phi$.
The frame operator is then defined as $S=\Phi \Phi^{T}.$
When $A=B$ the frame is called tight,  in which case  the frame operator is $S=AI$ where $I$ denotes the $N\times N$ identity matrix. 


\subsection{Scalable Frames}
Scalable frames were introduced in \cite{SF2013,SFaCG2013} as a method to convert a non tight frame into a tight one.
More precisely:
\begin{definition}\label{eq6}
Let $M \geq N$ be given.
A frame $\Phi=\{\varphi_k\}_{k=1}^{M} \subset \R^N$ is scalable if there exist a subset $\Phi_J=\{\varphi_k\}_{k\in J}$ with $J\subseteq \{1, 2, \hdots, M\}$, and positive scalars $\{x_k\}_{k\in J}$ such that the system $\widetilde{\Phi}_{J}=\{x_{k}\varphi_k\}_{k\in J}$ is a tight frame for $\R^N$.
\end{definition}

Let $\Phi=\{\varphi_{k}\}_{k=1}^{M} \subset \R^N$ be a frame.
Then the analysis operator of the scaled frame $\{x_{k}\varphi_{k}\}_{k=1}^{M}$ is given by $X\Phi^{T}$, where $X$ is the diagonal matrix with the values $x_{k}$ on its diagonal.
Hence, the frame $\Phi$ is scalable if and only if there exists a diagonal matrix $X=\text{diag}(x_k)$, with $x_k\geq 0$ such that
\begin{equation}\label{eq7}
\widetilde{S} = \Phi X^{T}X \Phi^{T} = \Phi X^2\Phi^T=AI.
\end{equation}
for some constant $A>0$. Without loss of generality we may assume that $A=1$--otherwise replace the diagonal matrix $X^2$ by $Y=X^2/A$.

One can covert ~\eqref{eq7} into a linear system of equations in $M$  unknowns: $x_k^2.$ To write out this linear system we need the following function:  $F:\R^{N}\rightarrow\R^{d}$ given by 
\begin{align*}
F(x) = [F_{0}(x),F_{1}(x),\dots,F_{N-1}(x)]^{T},\\ \\
F_{0}(x) = 
\begin{bmatrix}x_1^2 - x_2^2 \\ x_1^2 - x_3^2 \\ \vdots \\ x_1^2 - x_N^2\end{bmatrix},
F_{k}(x) = 
\begin{bmatrix}x_kx_{k+1} \\ x_kx_{k+2} \\ \vdots \\ x_kx_{N}\end{bmatrix}
\end{align*}
and $F_{0}(x)\in\R^{N-1}$, $F_{k}(x)\in\R^{N-k}$, $k=1,2,\dots,N-1$, where $d:= \frac{(N-1)(N+2)}{2}$.
Let $F(\Phi)$ be the $d\times M$ matrix  given by $$F(\Phi) = (F(\varphi_{1}) \,\, F(\varphi_{2}) \,\, \dots \,\, F(\varphi_{M})).$$



In this setting we have  the following solution to the scalability problem:

\begin{proposition}\label{prop1}\cite[Proposition~3.7]{SFaCG2013}
A frame $\Phi=\{\varphi_k\}_{k=1}^M \subset \R^{N}$  is scalable if and only if there exists a non-negative $u\in\ker F(\Phi)\backslash\{0\}$.
\end{proposition}
%
\subsection{Mathematical Programming and Duality}
Our main goal is to find a non-negative nontrivial vector in the null space of $F(\Phi)$ using some optimization methods. For this reason we recall some notions from 
duality theory in mathematical programming.
For a more robust treatment of duality theory applied to linear programs, we refer to the standard texts by S. Boyd, L. Vandenberghe \cite{CO2004} and D. Bertsimas, J. Tsitsiklis \cite{ItLO1997}.
Recall that the Primal and Dual linear mathematical programming problems are defined, respectively, as follows:
\begin{align*}
\text{minimize: }& c^{T}x\\
\text{subject to: } & Ax=b\\
                    & x\succeq0.
\end{align*}
\begin{align*}
\text{maximize: }&b^{T}y\\
\text{subject to: }&A^{T}y \leq c \\
&y\in\R^{N}.
\end{align*}

\begin{theorem}[Strong Duality]
If either the primal or dual problem has a finite optimal value, then so does the other.
The optimal values coincide, and optimal solutions to both the primal and dual problems exist.
\end{theorem}

\begin{theorem}[Complimentary Slackness]
Let $x^{*}$ and $y^{*}$ be feasible solutions to the primal and dual problems respectively.
Let $A$ be an $N$ by $M$ matrix, where $A_{j}$ denotes the $j$th column and $a_{i}$ denotes the $i$th row of $A$.
Then $x^{*}$ and $y^{*}$ are optimal solutions to their respective problems if and only if
\begin{equation*}
y_{i}(a_{i}\cdot x - b_{i}) = 0 \text{ for all }i=1,\dots,N,
\end{equation*}
and
\begin{equation*}
x_{i}(c_{j} - y^{T}A_{j}) = 0 \text{ for all }j=1,\dots,M.
\end{equation*}
\end{theorem}

\section{Reformulation of the Scalability Problem as an Optimization Problem}
This section establishes the equivalence of generating a scaling matrix $X$ and solving an optimization problem of a generic  convex objective function.
More specifically, we shall phrase the scalability problem as a linear and convex programming problem.

First consider the sets $\mathcal{S}_{1}$ and $\mathcal{S}_{2}$ given by 
$$\mathcal{S}_{1} := \{u\in\R^{M} \,|\, F(\Phi)u=0 \,,\, u\succeq0 \,,\, u\neq0\},$$
and
$$\mathcal{S}_{2} := \{v\in\R^{M} \,|\, F(\Phi)v=0 \,,\, v\succeq0 \,,\, \|v\|_{1}=1\}.$$
$\mathcal{S}_{1}$ is a subset of the null space of $F(\Phi)$, and  each $u \in \mathcal{S}_1$ is associated a scaling matrix $X_{u}$, defined as
$$X_{u} := (X_{ij})_{u} = \left\{\begin{matrix}\sqrt{u_{i}} & \text{ if } i=j \\ 0 & \text{ otherwise.}\end{matrix}\right.$$
$\mathcal{S}_{2}\subset \mathcal{S}_{1}\cap B_{\ell^{1}}$  where $B_{\ell^1}$ is the unit ball under the $\ell^1$ norm. 

We observe that  a frame $\Phi =\{\varphi_{k}\}_{k=1}^{M}\subset\R^{N}$ is scalable if and only if 
there exists a scaling matrix $X_{u}$ with $u\in\mathcal{S}_{1}$. Consequently, one can associate to $X_u$  a scaling matrix $X_{v}$ with $v\in\mathcal{S}_{2}$.
The normalized set $\mathcal{S}_{2}$ ensures that the constraints in the optimization problems to be presented are convex.

\begin{theorem}\label{thm1}
Let $\Phi =\{\varphi_{k}\}_{k=1}^{M}\subset\R^{N}$ be a frame, and let $f:\R^{M}\rightarrow\R$ be a convex function.
Then the program
\begin{align}\label{thm1_1}
(\mathcal{P}) \,\text{minimize:} &\, f(u)\\
              \text{subject to:} &\, F(\Phi)u=0\notag\\
                        & \|u\|_{1}=1\notag\\
                        & u \succeq0\notag
\end{align}
has a solution if and only if the frame $\Phi$ is scalable.
\end{theorem}
\begin{proof}
Any feasible solution $u^{*}$ of $\mathcal{P}$ is contained in the set $\mathcal{S}_{2}$, which itself is contained in $\mathcal{S}_{1}$, and thus corresponds to a scaling matrix $X_u$. 

Conversely, any $u\in\mathcal{S}_{1}$ can be mapped to a $v\in\mathcal{S}_{2}$ by appropriate scaling factor.
This provides an initial feasible solution to $\mathcal{P}$, and as $f$ is convex and the constraints are convex and bounded, there must exist a minimizer of $\mathcal{P}$.
\end{proof}
Theorem~\ref{thm1} is very general in that the convex objective function $f$ can be chosen so as the resulting frame has some desirable properties. We now consider certain interesting examples of objective functions $f$.
These examples can be  related to the sparsity (or lack thereof) of the desired solution. 
Using a linear objective function promotes sparsity, while barrier objectives promote dense solutions (small number of zero elements in $u$).

\subsection{Linear Program Formulation}
Assume that the objective function in \eqref{thm1_1} is given by 
$f(u) := a^{T}u$
for some coefficient vector $a\in\mathbb{R}^{M}\backslash \{0\}$.
Our program $\mathcal{P}$ now becomes

\begin{align}\label{GDC_1}
(\mathcal{P}_{1}) \text{ minimize: }& a^{T}u\\
              \text{subject to: } & F(\Phi)u=0\notag\\
                                  & \|u\|_{1}=1\notag\\
                                  & u\succeq0\notag.
\end{align}

Choosing the coefficients $a$  independently of the variables $u$, results in a linear program.
For example, the choice  $a_{i} = 1$ for all $i$ result in a program to minimize the $\ell^{1}$ norm of $u$.
Another, more useful choice of coefficients is $a_{i} = \frac{1}{\|F(\varphi_{i})\|_{2}}$.
Under this regime, a higher weight is given to the frame elements with smaller norm (which further encourages sparsity).

One of the advantages of linear programs is that they admit a strong dual formulation. To the  primal problem
$\mathcal{P}_{1}$ corresponds  the following  dual problem $\mathcal{P}_{2}$.

\begin{proposition}\label{GDC3}
Let $\Phi =\{\varphi_{k}\}_{k=1}^{M}\subset\R^{N}$ be a frame.
The program
\begin{align*}
  (\mathcal{P}_{2})\text{ maximize: }& w\\
\text{subject to: } & [F(\Phi)^{T} \,\, \mathbbm{1}]\begin{bmatrix}v \\ w\end{bmatrix} \leq a\\
                    & w\in\mathbb{R} \,,\, v\in\mathbb{R}^{d}
\end{align*}
is the strong dual of $\mathcal{P}_{1}$.
\end{proposition}
\begin{proof}
This result follows exactly from the construction of dual formulations for linear programs.
The primal problem can be formulated as follows:
\begin{align*}
\text{minimize: }&\sum_{i=1}^{M}a_{i}u_{i}\\
\text{subject to: }&F(\Phi)u=0\\ 
&\sum_{i=1}^{M}u_{i}=1\\
&u\succeq0.
\end{align*}
The strong dual of this problem is:
\begin{align*}
\text{maximize: }&w\\
\text{subject to: }&[F(\Phi)^{T} \,\, \mathbbm{1}]\begin{bmatrix}v \\ w\end{bmatrix}\leq a.
\end{align*}
\end{proof}

Numerical optimization schemes, in many cases, consist of a search for an initial feasible solution, and then a search for an optimal solution.
In analyzing the linear program formulation $\mathcal{P}_{1}$, we notice that we either have an optimal solution or the problem is infeasible, but there is no case when the problem is unbounded (due to the bounding constraint $\|u\|_{1}=1$).

The dual problem has the property that it either has an optimal solution, or is unbounded (from duality).
Consequently,  for any frame $\Phi$, $w=\min\{a\}$ and $v=0$ is always a feasible solution to the dual problem. This removes the requirement that an initial solution be found \cite{ItLO1997}.


\subsection{Barrier Formulations}
A sparse solution to the linear program produces a frame in which the frame elements corresponding to the zero coefficients are removed.
In contrast, one may wish to have a full solution, that is, one may want to retain all of, or most of, the frame vectors. To enforce this property, we  use a barrier objective.



\begin{proposition}\label{GDC4}
Let $\Phi =\{\varphi_{k}\}_{k=1}^{M}\subset\R^{N}$ be a frame, and define $0\leq\epsilon\ll1$.
If the problem
\begin{align}\label{GDC4_1}
(\mathcal{P}_{3}) \text{ maximize: }&\sum_{i=1}^{M}\ln(u_{i}+\epsilon)\\
\text{subject to: }&F(\Phi)u=0\notag\\ 
&\|u\|_{1}=1\notag\\
&u\geq0.\notag
\end{align}
has a feasible solution $u^{*}$ with a finite objective function value, then the frame $\Phi$ is scalable, and the scaling 
matrix $X$ is a diagonal operator where the elements are the square-roots of the feasible solution $u^{*}$.
Moreover, for $\epsilon=0$, if a solution $u^{*}$ exists, all elements of $u^{*}$ are strictly positive.
\end{proposition}
\begin{proof}
Assume $u^{*}$ is a feasible solution to \eqref{GDC4_1} with $0<\epsilon\ll1$ and the objective function finite.
Then from Theorem \ref{thm1}, we have that the frame $\Phi$ is scalable.
Now assume $\epsilon=0$.
If one of the variables $u_{i}$ were zero, then the objective function would have a value of $-\infty$.
Since we assume the function is finite, this cannot be the case.
A negative value for $u_{i}$ would result in the objective function value being undefined, this also cannot be the case due to the finite objective.
Therefore, $u_{i}$ must be positive for all $i$.
\end{proof}

\noindent An alternative barrier is the maximin objective.

\begin{proposition}\label{GDC5}
Let $\Phi =\{\varphi_{k}\}_{k=1}^{M}\subset\R^{N}$ be a frame.
If the problem
\begin{align}\label{GDC5_1}
(\mathcal{P}_{4}) \text{ maximize: }&\min_{i=1,\dots,M}\{u_{i}\}\\
\text{subject to: }&F(\Phi)u=0\notag\\ 
&\|u\|_{1}=1\notag\\
&u\geq0.\notag
\end{align}
has a feasible solution $u^{*}$ with a finite objective function value, then the frame $\Phi$ is scalable, and the scaling 
matrix $X$ is a diagonal operator where the elements are the square-roots of the feasible solution $u^{*}$.
Moreover, a solution exists with positive elements if and only if the solution produced by solving this problem has positive elements.
\end{proposition}
\begin{proof}
To show this, we shall rewrite this problem as a linear program.
\begin{align}
 \text{ maximize: }& t\\
\text{subject to: }& F(\Phi)u=0\notag\\
& \sum_{i=1}^{M}u_{i}=1\notag\\
& t\leq u_{i}\notag\\
& t>0 \,,\, u\succeq0\notag.
\end{align}
Here, $t$ is an auxiliary variable, taken to be the minimum element of $u$.
This linear program can be solved to optimality.
Moreover, as this problem is convex, the optimum achieved is global.
If the objective function at optimality has a value of 0, then there can exist no solution with all positive coefficients.
\end{proof}

\section{Augmented Lagrangian Method}
To efficiently solve these convex formulations, we employ the method of augmented Lagrangians. Rewriting norms as matrix/vector products we are interested in 
%
%
\begin{align*}
\text{ minimize: }&\, u^{T}Iu\\
              \text{subject to: } &\, F(\Phi)u=0\\
                        &\, \1^{T}u=1\\
                        &\, u \succeq0.
\end{align*}
For notational convenience, we denote $L$ and $b$ to be
\begin{align*}
\begin{bmatrix}F(\Phi)\\\1^{T} \end{bmatrix} \hspace{1em} &\text{and} \hspace{1em}\begin{bmatrix}0\\ 1\end{bmatrix}
\end{align*}
respectively.
The $\ell^{2}$ problem is now
\begin{align*}
\text{ minimize: }&\, u^{T}Iu\\
              \text{subject to: } &\, Lu=b\notag\\
                        &\, u \succeq0.\notag
\end{align*}
To solve this problem, the augmented Lagrangian $\Lg$ is formed,
\begin{align*}
\Lg =&\, u^{T}Iu + \langle\mu , Lu-b\rangle + \dfrac{\lambda}{2}\|Lu - b\|^{2}_{2}\\
=&\, u^{T}Iu + \mu^{T}Lu - \mu^{T}b + \dfrac{\lambda}{2}(u^{T}L^{T}Lu - 2u^{T}L^{T}b + b^{T}b).
\end{align*}
This function is minimized through satisfying the first-order condition, $\nabla \Lg = 0$.

The gradient of the Lagrangian with respect to $u$ is solved through standard calculus-based methods.
The Lagrangian with respect to the dual variables $\mu$ and $\lambda$ is linear, which we optimize through gradient descent.
The tuning parameter $\eta$ denotes the scaling of the descent direction.

\begin{align*}
\nabla_{u}\Lg =&\, 2Iu + L^{T}\mu + \lambda L^{T}Lu - \lambda L^{T}b.\\
    0 =&\, (2I + \lambda L^{T}L)u + L^{T}\mu - \lambda L^{T}b.\\
    (2I + \lambda L^{T}L)u =&\,  \lambda L^{T}b - L^{T}\mu.
\end{align*}
Dividing the equation by $\lambda$, we have
\begin{align*}
\left(\dfrac{2}{\lambda}I + L^{T}L\right)u =&\,  L^{T}b - L^{T}\dfrac{\mu}{\lambda}.\\
\left(\dfrac{2}{\lambda}I + L^{T}L\right)u =&\,  L^{T}\left(b - \dfrac{\mu}{\lambda}\right).\\
u =&\, \left(\dfrac{2}{\lambda}I + L^{T}L\right)^{-1}L^{T}\left(b - \dfrac{\mu}{\lambda}\right).
\end{align*}
The dual variables have the following gradients,
\begin{align*}
\nabla_{\mu}\Lg =&\, Lu - b.\\
\nabla_{\lambda}\Lg =&\, \dfrac{1}{2}\langle Lu - b,Lu - b \rangle.\\
\end{align*}
And forming the gradient descent algorithm with $\eta$, results in
\begin{algorithm}[H]\caption{Gradient Descent w.r.t. $\mu$}\begin{algorithmic}
\WHILE{not converged}
\STATE{$\mu_{k+1} \gets \mu_{k} - \eta\cdot(Lu - b)$}
\ENDWHILE
\end{algorithmic}\end{algorithm}
\begin{algorithm}[H]\caption{Gradient Descent w.r.t. $\lambda$}\begin{algorithmic}
\WHILE{not converged}
\vspace{.3em}
\STATE{$\lambda_{k+1} \gets \lambda_{k} - \dfrac{\eta}{2}\cdot\langle Lu - b,Lu - b \rangle$}
\ENDWHILE
\end{algorithmic}\end{algorithm}
Lastly, to retain the non-negativity of the solution, we project the current solution onto $\R_{+}$.
This is accomplished by setting any negative values in the solution to 0 (thresholding).
We shall denote this $\mathcal{P}_{+}(\cdot)$.
Forming the full augmented Lagrangian scheme, we now have the complete $\ell^{2}$ derivation.
\begin{algorithm}[H]\caption{Full Augmented Lagrangian Scheme ($\ell^{2}$)}\begin{algorithmic}
\WHILE{not converged}
\vspace{.4em}
\STATE{$v_{k+1} \gets \left(\dfrac{2}{\lambda_{k}}I + L^{T}L\right)^{-1}L^{T}\left(b - \dfrac{\mu_{k}}{\lambda_{k}}\right)$}
\vspace{.4em}
\STATE{$u_{k+1} \gets \mathcal{P}_{+}(v_{k+1})$}
\STATE{$\mu_{k+1} \gets \mu_{k} - \eta\cdot(Lu_{k+1} - b)$}
\vspace{.4em}
\STATE{$\lambda_{k+1} \gets \lambda_{k} - \dfrac{\eta}{2}\cdot\langle Lu_{k+1} - b,Lu_{k+1} - b \rangle$}
\ENDWHILE
\end{algorithmic}\end{algorithm}

\section{Numerical Examples}
The following numerical tests are intended to illustrate our methods by scaling frames generated from  Gaussian distributions. In particular, throughout this section, we identified  (random) frames $\Phi$ with (full rank) random $N\times M$  matrices whose elements are i.i.d., drawn from a Gaussian distribution with zero mean and unit variance.


The first set of figures are intended to give a representation of how the scaling affects frames in $\R^{2}$.
A number of Gaussian random frames are generated in MatLab, and a scaling process is performed by solving one of the optimization problems above (the specific program used is noted under the figures).
The Gaussian frame is first normalized to be on the unit circle.
The (blue$\backslash$circle) vectors correspond to the original frame vectors, and the (red$\backslash$triangle) vectors  represent the resulting scaled frame.

\begin{figure}[H]
\centering
\includegraphics[scale=.28]{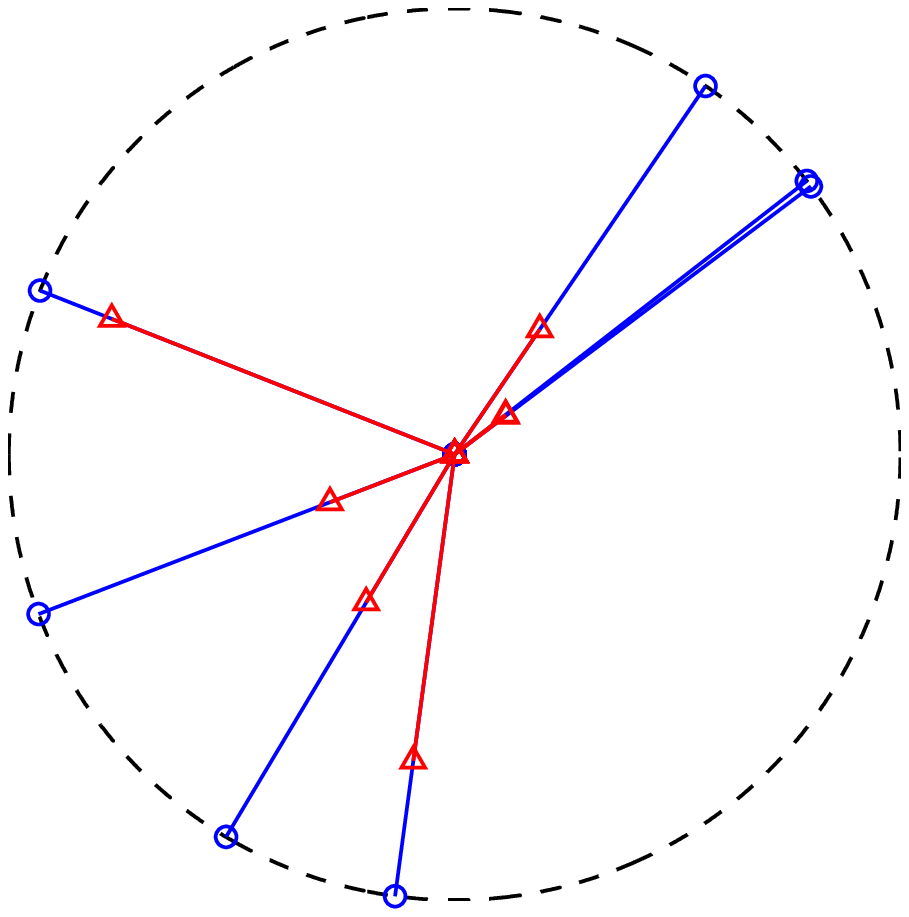}
\includegraphics[scale=.28]{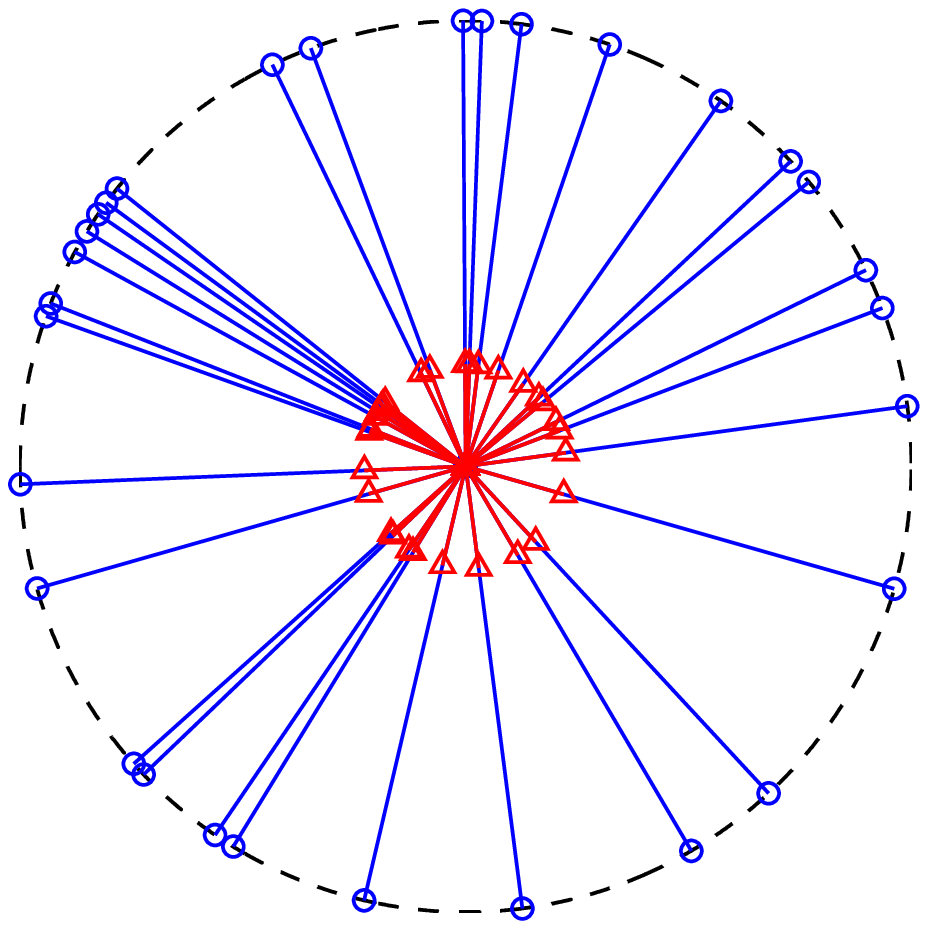}
\caption{These examples display the effect of scaling frames in $\R^2$. The frames are sized $M=7$ (Left) and $M=30$ (Right), and were scaled using the Augmented Lagrangian Scheme.
The left figure shows that scalings favor isolated frame elements.
The right figure shows that as the frame elements fill the space, the scalings become more normalized.}
\label{fig1}
\end{figure}

\begin{figure}[H]
\centering
\includegraphics[scale=.28]{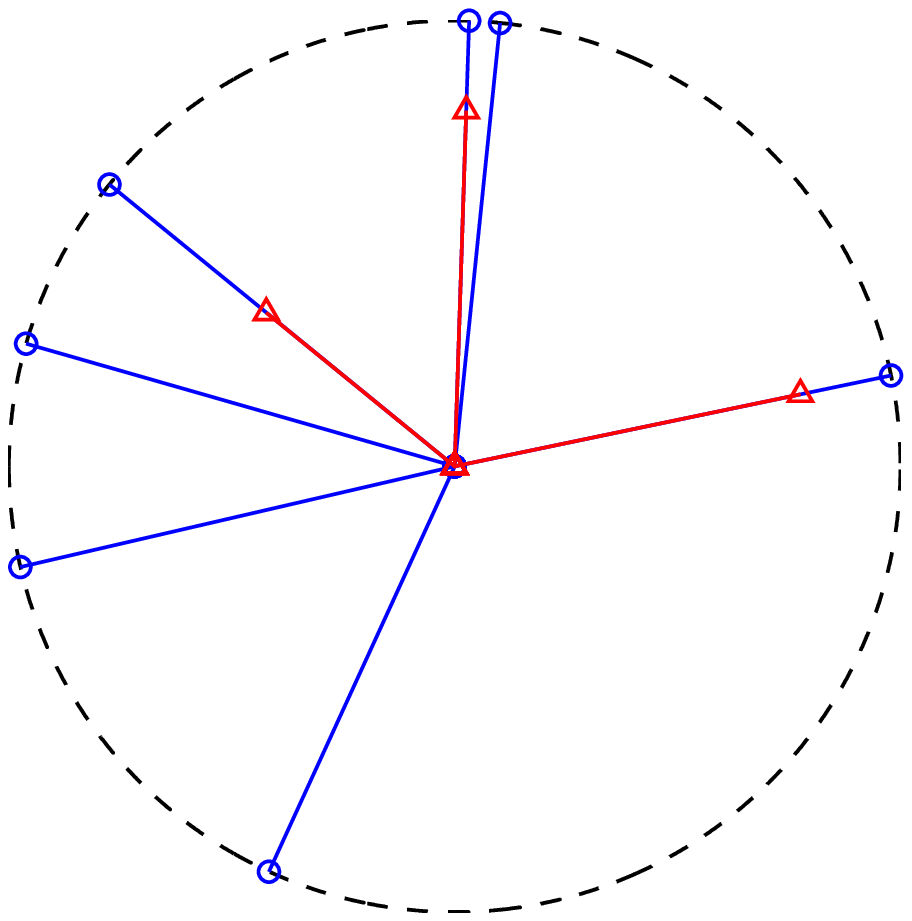}
\includegraphics[scale=.28]{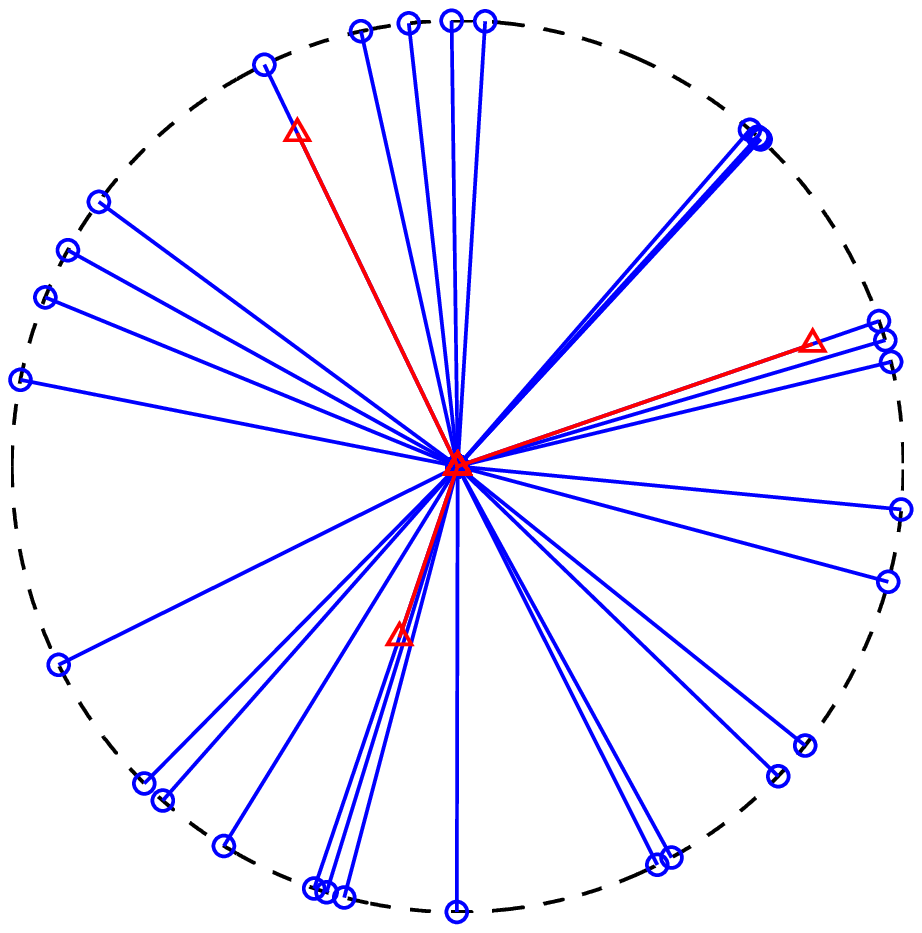}
\caption{These examples illustrate scalable frames with a small number of positive weights. The frames are sized $M=7$ (Left) and $M=30$ (Right), and were scaled using linear programming formulation $\mathcal{P}_{1}$ (more specifically, the Simplex algorithm).
These two example show that for frames of low (Left) and high (Right) redundancy, sparse solutions are possible and seemingly unrelated to the number of frame elements.}
\label{fig2}
\end{figure}

\begin{figure}[H]
\centering
\includegraphics[scale=.28]{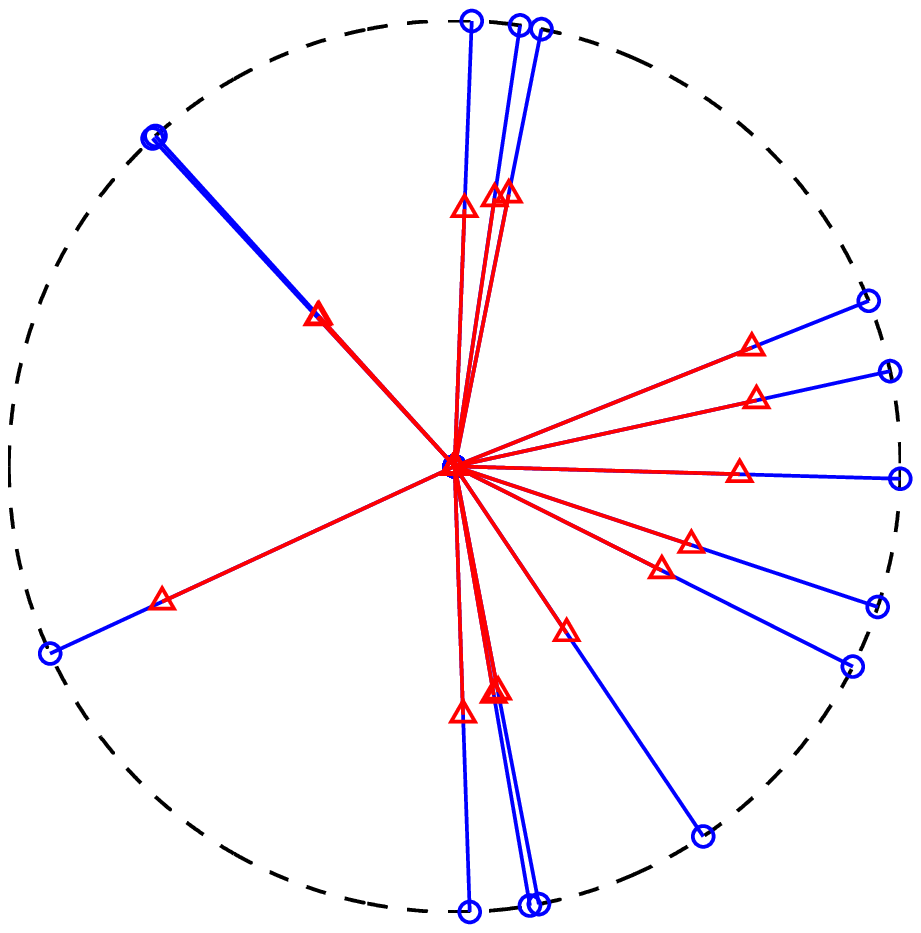}
\includegraphics[scale=.28]{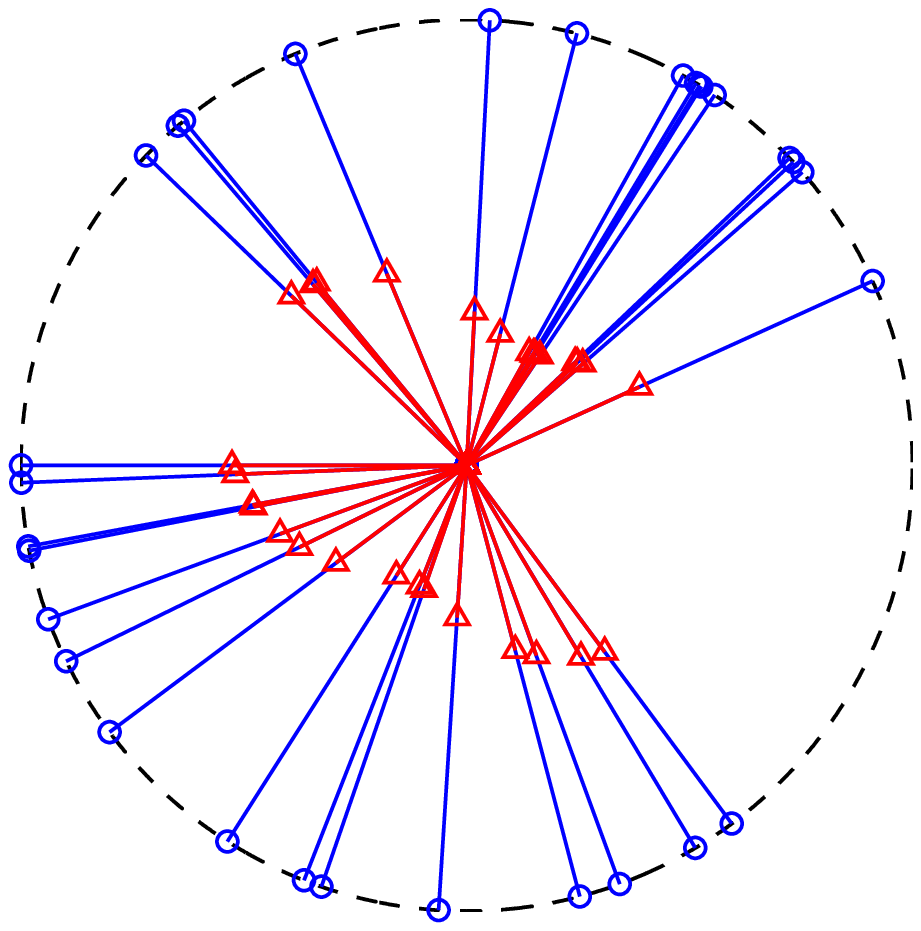}
\caption{These frames show full scaling results from the log-barrier method.
The frames are sized $M=15$ (Left) and $M=20$ (Right), and as was mentioned in figure \ref{fig1}, the scalings favor isolated frame elements.}
\label{fig3}
\end{figure}

The next tables illustrate  the sparsity that is achieved in scaling a frame. That is, using the linear program, we  present the average number of non-zero frame elements retained over 100 trials.
Our data seem to suggest the existence of  a minimum number of frame elements required to perform a scaling, and this number seems to  depend only on the dimension of the underlying space. This phenomenon  should be compared to the estimates on the probability that a frame of $M$ vectors in $R^N$  is scalable  that were obtained in \cite[Theorem 4.9]{MoS2014}. 

\begin{table}[H]
\centering
\begin{tabular}{|c|c|c|c|c|c|c|c|c|}
\multicolumn{9}{c}{Sparsity Test Results [Gaussian Frames]}\\
\hline
$N\backslash M$ &   3 &   4 &   5 &  10 &  20 &  30 &  40 &  50 \\\hline
2               & 3 & 3 & 3 & 3 & 3 & 3 & 3 & 3 \\
3               & - & - & - & 6.01 & 6 & 6 & 6 & 6 \\
4               & - & - & - & 10.1 & 10.12 & 10.1 & 10 & 10 \\
5               & - & - & - & - & 15.08 & 15.12 & 15.11 & 15.05 \\
\hline
\end{tabular}
\caption{The average number of frame elements retained after scaling using the linear program formulation.
Entries with a ``-" imply that the proportion of scalable frames in the space is too small for practical testing.}
\label{tab1}
\end{table}

\begin{table}[H]
\centering
\begin{tabular}{|c|c|c|c|c|c|c|}
\multicolumn{7}{c}{Sparsity Test Results [Gaussian Frames]}\\
\hline
$N\backslash M$  &  150 &  200 & 250 & 500 & 750 & 1000\\\hline
10               & 56.06 & 56.02 & 55.72 & 55.57 & 55.66 &  55.6\\
15               & - & - & 123.76 & 123.98 & 123.37 &  123.1\\
20               & - & - & - & 217.6 & 218.6 &  219.45\\
25               & - & - & - & - & - & 338.67\\
\hline
\end{tabular}
\caption{The average number of frame elements retained after scaling using the linear program formulation.
Entries with a ``-" imply that the proportion of scalable frames in the space is too small for practical testing.}
\label{tab2}
\end{table}

Observe that the results presented in tables \ref{tab1} and \ref{tab2} show an interesting trend. The average number of elements required to scale a frame appears to be $$d+1=\frac{(N-1)(N+2)}{2}+1 = \dfrac{N(N+1)}{2}.$$
The linear system being solved during the Simplex method has dimensions $d+1\times M$, and attempts to find a non-negative solution in $\R^{M}$.
It seems unlikely that this solution can be found using less than $d+1$ frame elements.

The final test presents the proportion of scalable Gaussian frames of a given size over 100 trials.
For testing, the number of frame vectors is determined   by the dimension of the underlying space.
For a frame in $\R^N$, the number of frame elements used ranges from $N+1$ to $4N^2$ (e.g. for $N=2$, the number of frame elements range from $M=3$ to $M=16$).
A Gaussian frame is generated of the required sizes and a scaling is attempted.
This is performed over a hundred trials, and the proportion of frames that were scalable was retained.

For each $N$, a plot of the proportions across frame size $M$ is presented.
To display these plots in a single figure, the independent variable $M$ is scaled to lie in the range (0,1).

\begin{figure}[H]
\centering
\includegraphics[scale=.6]{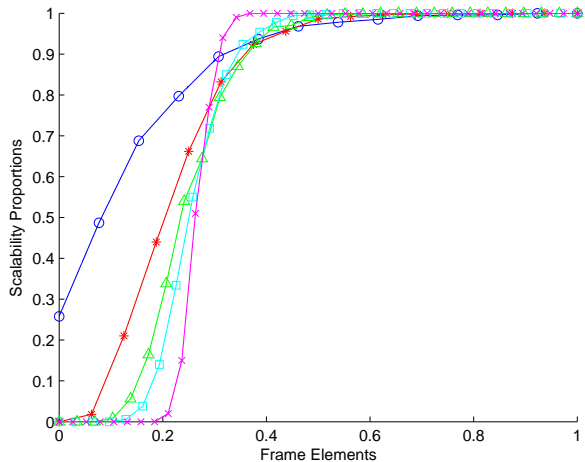}
\caption{Each graph in this figure gives the proportion of scalable frames generated after 100 trials.
The sizes of the frames range from $N+1$ to $4N^2$.
To fit the graphs in a single figure, the range of each figure is scaled to be from 0 to 1.
The frame dimensions are $N=2$ (Blue$\backslash$Circle), $N=3$ (Red$\backslash$Star), $N=4$ (Green$\backslash$Triangle), $N=5$ (Cyan$\backslash$Box), and $N=10$ (Magenta$\backslash$x).}
\label{fig4}
\end{figure}

\section{Acknowledgments}
Chae Clark would like to thank the Norbert Wiener Center for Harmonic Analysis and Applications for its support during this research. Kasso Okoudjou was partially supported  by a RASA from the Graduate School of UMCP and by a grant from the Simons Foundation ($\# 319197$ to Kasso Okoudjou).

\nocite{MoS2014,FF2012,SF2013}
\nocite{SFaCG2013,ADbBP1998,ItLO1997}
\nocite{CO2004,OUS1991,MC2012}
\nocite{COwSIN2011,DbLP2005}
\bibliography{references}
\bibliographystyle{plain}
\end{document}